\documentclass[11pt,centertags]{amsart}
\usepackage{amscd}

\usepackage{easybmat}

\usepackage{mathrsfs}
\usepackage{amsfonts}
\usepackage{color}
\usepackage{pifont}
\usepackage{upgreek}
\usepackage{bm}
\usepackage{hyperref}
\usepackage{shorttoc}

\usepackage{amsmath,amstext,amsthm,a4,amssymb,amscd}
\usepackage[mathscr]{eucal}
\usepackage{mathrsfs}
\usepackage{epsf}
\numberwithin{equation}{section}

\def\p{\partial}
\def\o{\overline}
\def\b{\bar}
\def\mb{\mathbb}
\def\mc{\mathcal}
\def\n{\nabla}

\def\t{\triangle}

\def\p{\partial}
\def\o{\overline}
\def\b{\bar}
\def\mb{\mathbb}
\def\mc{\mathcal}
\def\n{\nabla}

\def\t{\triangle}

\newtheorem{thm}{Theorem}[section]
\newtheorem{lemma}[thm]{Lemma}
\newtheorem{prop}[thm]{Proposition}

\theoremstyle{definition}

\theoremstyle{definition}
\newtheorem{defn}[thm]{Definition}
\newcommand{\be}{\begin{eqnarray}}
\newcommand{\ee}{\end{eqnarray}}

\newcommand{\comment}[1]{}

\usepackage{fancyhdr}
\begin{document}

\title{Compact K\"{a}hler manifolds with positive orthogonal bisectional curvature}

\author[Huitao Feng]{Huitao Feng$^1$}
\author[Kefeng Liu]{Kefeng Liu$^2$}
\author[Xueyuan Wan]{Xueyuan Wan$^3$}
\address{Huitao Feng: Chern Institute of Mathematics \& LPMC,
Nankai University, Tianjin, China}

\email{fht@nankai.edu.cn}

\address{Kefeng Liu: Department of Mathematics, UCLA, Los Angeles, USA}

\email{liu@math.ucla.edu}

\address{Xueyuan Wan: Chern Institute of Mathematics,
Nankai University, Tianjin, China}

\email{xywan@mail.nankai.edu.cn}

\thanks{$^1$~Partially supported by NSFC (Grant No. 11221091, 11271062, 11571184)}

\thanks{$^3$~Partially supported by NSFC (Grant No. 11221091,11571184)
and the Ph.D. Candidate Research Innovation Fund of Nankai University}

\begin{abstract}
In this short note, using Siu-Yau's method \cite{Siu}, we give a new proof that any $n$-dimensional compact K\"{a}hler manifold with positive orthogonal
 bisectional curvature must be biholomorphic to $\mb{P}^n$.
\end{abstract}
\maketitle

\setcounter{section}{-1}

\section{Introduction} \label{s0}

In the celebrated paper \cite{Siu}, Siu and Yau presented a differential geometric proof of the famous Frankel conjecture in K\"{a}hler geometry, which states that a compact K\"{a}hler manifold $M^n$ with positive bisectional curvature must be biholomorphic to $\mb{P}^n$. Using the method of Siu-Yau, Seaman \cite{Sea} in 1993 proved that any compact K\"{a}hler manifold $M^n$ with positive curvature on totally isotropic 2-planes must be biholomorphic to $\mb{P}^n$. On the other hand, there is also a concept of orthogonal bisectional curvature (cf. Definition \ref{defn.1} in this paper) for K\"{a}hler manifolds of dimension $n\geq 2$, which was introduced by Cao and Hamilton in the late 1980's. Note that the condition of positive orthogonal bisectional curvature is weaker than both the conditions of positive bisectional curvature and positive curvature on totally isotropic 2-planes for a K\"{a}hler manifold $M^n$ with $n\geq 2$. Hence a natural question is whether a compact K\"{a}hler manifold $M^n$ ($n\geq 2$) with positive orthogonal bisectional curvature is biholomorphic to $\mb{P}^n$. Inspired by an observation of Cao and Hamilton in an unpublished work\footnote{\text{Their proof has since appeared}, see Theorem 2.3 of \cite{Cao}.} that the nonnegativity of the orthogonal bisectional curvature is preserved under the K\"{a}hler-Ricci flow, Chen \cite{Chen} in 2007 gave a positive answer, under the extra condition $c_1(M)>0$, to this question by using the K\"{a}hler-Ricci flow method. Chen further asked whether this extra condition could be dropped. Later, H. Gu and Z. Zhang \cite{Gu} in 2010 showed that $c_1(M)>0$ holds for any compact K\"{a}hler manifold $M^n$ with positive orthogonal bisectional curvature. Therefore, by combining the results of Chen and Gu-Zhang, one has
\begin{thm}\label{thm.01} Let $M$ be an $n$-dimensional compact K\"{a}hler manifold with positive orthogonal bisectional curvature, then $M$ is biholomorphic to $\mb{P}^n$.
\end{thm}

In this note, we will give a new proof of the above theorem by using Siu-Yau's method \cite{Siu}, which answers a question of Chen raised in \cite{Chen}. Note that Chen's proof depends heavily on the K\"{a}hler-Ricci flow techniques. More precisely, by assuming $c_1(M)>0$, he can show that a K\"{a}hler metric with positive orthogonal bisectional curvature flows to a metric of positive holomorphic bisectional curvature, and then he got a proof of Theorem 0.1 by using Siu-Yau's result. Compared with Chen's proof, our proof is more direct and geometric.

This note is organized as follows. In Section 1, by computing directly the second variation of the energy of maps $f:\mb{P}^1\to (M,h)$, we prove that an energy minimal map $f$ must be holomorphic or conjugate holomorphic when $(M,h)$ is a compact K\"{a}hler manifold with positive orthogonal bisectional curvature, which is the key step of our proof of Theorem 0.1. In Section 2, we will complete the proof of Theorem 0.1 by using Siu-Yau's method \cite{Siu}.

\

{\bf Acknowledgements.} The authors would like to thank Professor Xiaokui Yang for his helpful suggestions in preparing this paper. The authors would like to thank the reviewers for their comments that help improve the paper.

\section{Complex Analyticity of Energy Minimizing Maps}

In this section, we first compute the first and second variations of the total energy of smooth maps $f:\mb{P}^1\to (M,h)$, and then prove the main result of this section that an energy minimal map $f$ must be holomorphic or conjugate holomorphic when $(M,h)$ is a compact K\"{a}hler manifold with positive orthogonal bisectional curvature. Note that Siu-Yau \cite{Siu} used $\bar{\partial}$-energy in their proof.

Let $\mb{P}^1$ be the complex projective space of complex dimension one with a fixed conformal structure $\omega$ and $M$ a compact K\"{a}hler manifold with a K\"{a}hler metric $h$. In local holomorphic coordinates on $\mb{P}^1$ and $M$, we write $\omega$ and $h$ as following respectively:
$$\omega=\lambda^2 dw\otimes d\b{w},$$
$$h=\sum_{\alpha,\beta=1}^n h_{\alpha\b{\beta}}dz^{\alpha}\otimes d\b{z}^{\beta}.$$
For any smooth map $f:(\mb{P}^1,\omega)\to (M,h)$, the energy of $f$ can be written, with respect to  $h$, as following:
\begin{align}\label{energy.1}
  E(f)=\int_{\mb{P}^1}\left\langle\frac{\p f}{\p w},\frac{\p f}{\p w}\right\rangle\sqrt{-1}dw\wedge d\b{w}=\int_{\mb{P}^1}\left(f^{\alpha}_w\o{f^{\beta}_w}+f^{\alpha}_{\b{w}}\o{f^{\beta}_{\b{w}}}\right)h_{\alpha\b{\beta}}\sqrt{-1}dw\wedge d\b{w},
\end{align}
where the summation convention is used, and
\begin{align}\label{f1}
f^{\alpha}_w=\frac{\p f^{\alpha}}{\p w},\quad f^{\alpha}_{\b{w}}=\frac{\p f^{\alpha}}{\p\b{w}},
\end{align}
\begin{align}\label{f2}
\frac{\p f}{\p w}:=f_*\left(\frac{\p}{\p w}\right)=f^{\alpha}_w\frac{\p}{\p z^{\alpha}}+\o{f^{\alpha}_{\b w}}\frac{\p}{\p\b{z}^{\alpha}}.
\end{align}
Let $\n^{Ch}$ be the Chern connection with respect to the K\"{a}hler  metric $h$. The Christoffel symbols $\Gamma^{\alpha}_{\beta\gamma}$ are defined by
$$\n^{Ch}_{\frac{\p}{\p z^{\beta}}}\frac{\p}{\p z^{\gamma}}=\Gamma^{\alpha}_{\beta\gamma}\frac{\p}{\p z^{\alpha}},$$
 and given by
 $$\Gamma^{\alpha}_{\beta\gamma}=h^{\b{\delta}\alpha}\frac{\p h_{\beta\b{\delta}}}{\p z^{\gamma}}.$$

Let $f(t):\mb{P}^1\to M$, $t\in\mb{C}$, $|t|<\epsilon$, be a family of smooth maps parametrized by an open disc in $\mb{C}$. Denote by $D$ the pullback connection $f(t)^*\n^{Ch}$, then by definition,
\begin{align}
\frac{D}{\p w}=\n^{Ch}_{f(t)_*\frac{\p}{\p w}},\quad \frac{D}{\p\b{w}}=\n^{Ch}_{f(t)_*\frac{\p}{\p\b{w}}},\quad  \frac{D}{\p t}=\n^{Ch}_{f(t)_*\frac{\p}{\p t}}, \quad\frac{D}{\p\b{t}}=\n^{Ch}_{f(t)_*\frac{\p}{\p\b{t}}}.
\end{align}

 Therefore,
\begin{align}
\begin{split}
  dE(f)\left(\frac{\p}{\p t}\right)&=\int_{\mb{P}^1}\left(\left\langle\frac{D}{\p t}\left(\frac{\p f}{\p w}\right),\frac{\p f}{\p w}\right\rangle+\left\langle\frac{\p f}{\p w},\frac{D}{\p\b{t}}\left(\frac{\p f}{\p w}\right)\right\rangle\right)\sqrt{-1}dw\wedge d\b{w}\\
  &=\int_{\mb{P}^1}\left(\left\langle\frac{D}{\p w}\left(\frac{\p f}{\p t}\right),\frac{\p f}{\p w}\right\rangle+\left\langle\frac{\p f}{\p w},\frac{D}{\p w}\left(\frac{\p f}{\p\b{t}}\right)\right\rangle\right)\sqrt{-1}dw\wedge d\b{w}\\
  &=-\int_{\mb{P}^1} \left(\left\langle\frac{\p f}{\p t},\frac{D}{\p\b{w}}\left(\frac{\p f}{\p w}\right)\right\rangle+\left\langle\frac{\p f}{\p t},\frac{D}{\p w}\left(\frac{\p f}{\p \b{w}}\right)\right\rangle\right)\sqrt{-1}dw\wedge d\b{w}\\
  &=-2\int_{\mb{P}^1}\left\langle\frac{\p f}{\p t},\frac{D}{\p\b{w}}\left(\frac{\p f}{\p w}\right)\right\rangle\sqrt{-1}dw\wedge d\b{w},
  \end{split}
\end{align}
where the first  identity uses the fact that $D$ is compatible with the metric, the second and fourth equalities are  by the torsion-freeness of $D$.

The harmonic map equation is
\begin{align}\label{energy.2}
  \frac{D}{\p\b{w}}\left(\frac{\p f}{\p w}\right)=0 \quad \,\text{or}\quad \frac{D}{\p w}\left(\frac{\p f}{\p\b{w}}\right)=0,
\end{align}
which is equivalent to
\begin{align}\label{harmonic}
  \frac{\p^2 f^{\alpha}}{\p w\p\b{w}}+\Gamma^{\alpha}_{\beta\gamma}f^{\gamma}_{\b{w}}f^{\beta}_w=0.
\end{align}

In the following, we compute the second variation of the energy (\ref{energy.1}). Note that similar formula has been derived by Moore from a different point of view (\cite{Moore} , (6)). Suppose that $f$ is a harmonic map, i.e. $f$ satisfies (\ref{energy.2}), then
\begin{align}\label{energy.3}
\begin{split}
  D^2E(f)\left(\frac{\p}{\p t},\frac{\p}{\p\b{t}}\right)&=-2\int_{\mb{P}^1}\left\langle \frac{\p f}{\p t},\frac{D}{\p t}\frac{D}{\p\b{w}}\left(\frac{\p f}{\p w}\right)\right\rangle\sqrt{-1}dw\wedge d\b{w}\\
  &=-2\int_{\mb{P}^1}\left(\left\langle\frac{\p f}{\p t}, R\left(\frac{\p f}{\p t},\frac{\p f}{\p\b{w}}\right)\frac{\p f}{\p w}\right\rangle
  +\left\langle \frac{\p f}{\p t},\frac{D}{\p\b{w}}\frac{D}{\p t}\left(\frac{\p f}{\p w}\right)\right\rangle\right)\sqrt{-1}dw\wedge d\b{w}\\
  &=2\int_{\mb{P}^1}\left(\left\langle\frac{D}{\p w}\frac{\p f}{\p t},\frac{D}{\p w}\frac{\p f}{\p t}\right\rangle
  -\left\langle R\left(\frac{\p f}{\p\b{t}},\frac{\p f}{\p w}\right)\frac{\p f}{\p\b{w}},\frac{\p f}{\p \b{t}}\right\rangle\right)\sqrt{-1}dw\wedge d\b{w}\\
  &\quad -2\int_{\mb{P}^1}\frac{\p}{\p w}\left\langle\frac{\p f}{\p t},\frac{D}{\p t}\frac{\p f}{\p w}\right\rangle\sqrt{-1}dw\wedge d\b{w}\\
  &=2\int_{\mb{P}^1}\left(\left\langle\frac{D}{\p w}\frac{\p f}{\p t},\frac{D}{\p w}\frac{\p f}{\p t}\right\rangle
  -\left\langle R\left(\frac{\p f}{\p\b{t}},\frac{\p f}{\p w}\right)\frac{\p f}{\p\b{w}},\frac{\p f}{\p \b{t}}\right\rangle\right)\sqrt{-1}dw\wedge d\b{w}\\
  &=2\int_{\mb{P}^1}\left(\left|\frac{D}{\p w}\left(\frac{\p f}{\p t}\right)\right|^2-\left\langle R\left(\frac{\p f}{\p\b{t}},\frac{\p f}{\p w}\right)\frac{\p f}{\p\b{w}},\frac{\p f}{\p \b{t}}\right\rangle\right)
  \sqrt{-1}dw\wedge d\b{w},
  \end{split}
\end{align}
where the second equality is by definition of  curvature operator
$$R\left(\frac{\p f}{\p\b{t}},\frac{\p f}{\p w}\right)=\frac{D}{\p\b{t}}\frac{D}{\p w}-\frac{D}{\p w}\frac{D}{\p\b{t}},$$
the fourth equality is given by the Stoke's Theorem.

 By (\ref{f2}) and
 $$\frac{\p f}{\p t}=f^{\alpha}_t\frac{\p}{\p z^{\alpha}}+\o{f^{\alpha}_{\b{t}}}\frac{\p}{\p\b{z}^{\alpha}},$$
 we see that the curvature term in (\ref{energy.3}) is given by
\begin{align}\label{energy.4}
\begin{split}
  \left\langle R\left(\frac{\p f}{\p\b{t}},\frac{\p f}{\p w}\right)\frac{\p f}{\p\b{w}},\frac{\p f}{\p \b{t}}\right\rangle
  &=R_{\alpha\b{\beta}\gamma\b{\delta}}\left(f^{\alpha}_{\b{t}}\o{f^{\beta}_{\b{w}}}
  -f^{\alpha}_{w}\o{f^{\beta}_{t}}\right)\left(f^{\gamma}_{\b{w}}\o{f^{\delta}_{\b{t}}}-f^{\gamma}_{t}\o{f^{\delta}_{w}}\right),
\end{split}
\end{align}
since $\langle R(X,Y)Z,W\rangle$ depends on $X,Y,Z$ linearly, and it depends on $W$ conjugate linearly,
where $ R_{\alpha\b{\beta}\gamma\b{\delta}}$ is given by
\begin{align}
  R_{\alpha\b{\beta}\gamma\b{\delta}}:=\left\langle R\left(\frac{\p}{\p z^{\alpha}},\frac{\p}{\p \b{z}^{\beta}}\right)\frac{\p }{\p z^{\gamma}},\frac{\p}{\p z^{\delta}}\right\rangle=-\frac{\p^2 h_{\gamma\b{\delta}}}{\p z^{\alpha}\p\b{z}^{\beta}}+h^{\b{\tau}\sigma}\frac{\p h_{\gamma\b{\tau}}}{\p z^{\alpha}}\frac{\p h_{\sigma\b{\delta}}}{\p\b{z}^{\beta}}.
\end{align}
\\

Before giving the main result, we first recall the definition of positive orthogonal bisectional curvature.

\begin{defn}\label{defn.1}
  A K\"{a}hler manifold $(M,h)$ of dimension $n\geq 2$ is said of positive orthogonal bisectional curvature if
\begin{align}\label{R}
R(X,\b{X},Y,\b{Y})>0
\end{align}
for any nonzero vectors $X,Y\in T^{1,0}M$ with $\langle X,Y\rangle=0$.
\end{defn}

The main result of this section is the following Theorem 1.2, which is also the key step in our proof of Theorem 0.1. The key point in the proof of Theorem 1.2 is that holomorphic and conjugate holomorphic sections are actually orthogonal to each other, which has been observed already in Siu \cite{Siu1} and Futaki \cite{Fut}.

\begin{thm}\label{thm.1}
 Let $(M,h)$ be a compact K\"{a}hler manifold of dimension $n\geq 2$ with positive orthogonal bisectional curvature. Then any energy minimizing map $f:\mb{P}^1\to M$ must be holomorphic or conjugate holomorphic.
\end{thm}

\begin{proof} Let $f$ be an energy minimizing map. If $f$ is neither holomorphic nor conjugate holomorphic, then we will get a contradiction. To show this, noticing that $\dim H^0(\mb{P}^1, T\mb{P}^1)=3$, we can take a nonzero holomorphic vector field $v\frac{\p}{\p w}$ of $\mb{P}^1$.
Then by the assumption that $f$ is neither holomorphic nor conjugate holomorphic, the following two vector fields of type $(1,0)$
\begin{align}\label{XY}
X:=\left[f_*\left(v\frac{\p}{\p w}\right)\right]^{(1,0)}=v\frac{\p f^{\alpha}}{\p w}\frac{\p}{\p z^{\alpha}},\quad Y:=\left[\o{f_*\left(v\frac{\p}{\p w}\right)}\right]^{(1,0)}=\b{v}\frac{\p f^{\alpha}}{\p\b{w}}\frac{\p}{\p z^{\alpha}}
\end{align}
are nonzero on $M$, where $[\bullet]^{(1,0)}$ denotes the $(1,0)$-part of a vector field. When there is no confusion, we also view $X,Y$ as two sections of the pull-back bundle $f^*TM$ over $\mb{P}^1$. Since $f$ is a harmonic map,
i.e., it satisfies (\ref{harmonic}), we have
  \begin{align}\label{1.3}
    \frac{D}{\p w}Y=\b{v}\frac{D}{\p w}\left(\frac{\p f^{\alpha}}{\p\b{w}}\frac{\p}{\p z^{\alpha}}\right)=\b{v}\left(\frac{\p^2 f^{\alpha}}{\p w\p\b{w}}+f^{\gamma}_w f^{\beta}_{\b{w}}\Gamma^{\alpha}_{\beta\gamma}\right)\frac{\p}{\p z^{\alpha}}=0.
  \end{align}
  Now we take such variation direction,
  \begin{align}
    \frac{\p f(t)}{\p t}|_{t=0}=Y.
  \end{align}
Therefore,
  \begin{align}
   \left(\frac{D}{\p w}\frac{\p f}{\p t}\right)|_{t=0}=\frac{D}{\p w}\left(\frac{\p f}{\p t}|_{t=0}\right)=\frac{D}{\p w}Y=0.
  \end{align}
Now by the second variation formula (\ref{energy.3}), we have
 \begin{align}\label{ef}
  \begin{split}
      \frac{\p^2 E(f)}{\p t\p\b{t}}&=-2\int_{\mb{P}^1}|v|^2R_{\alpha\b{\beta}\gamma\b{\delta}}f^{\alpha}_w\o{f^{\beta}_{\b{w}}}
      f^{\gamma}_{\b{w}}\o{f^{\delta}_w}\sqrt{-1}dw\wedge d\b{w}\\
      &=-2\int_{\mb{P}^1}|v|^{-2} R(X,\b{X},Y,\b{Y})\sqrt{-1}dw\wedge d\b{w},
  \end{split}
\end{align}
where $|v|^{-2}R(X,\b{X},Y,\b{Y})$ is taken to be zero at the zero points of $v$.

By (\ref{harmonic}) and (\ref{XY}), we have
\begin{align}\label{1.2}
  \frac{D}{\p\b{w}}X=v\left(\frac{\p^2 f^{\alpha}}{\p w\p\b{w}}+f^{\gamma}_w f^{\beta}_{\b{w}}\Gamma^{\alpha}_{\beta\gamma}\right)\frac{\p}{\p z^{\alpha}}=0.
\end{align}

By (\ref{1.3}) and (\ref{1.2}), we get
  \begin{align}
    \frac{\p}{\p\b{w}}\langle X,Y\rangle=\left\langle\frac{D}{\p\b{w}}X,Y\right\rangle+\left\langle X,\frac{D}{\p w}Y\right\rangle=0,
  \end{align}
which implies that  $\langle X,Y\rangle$ is a holomorphic function on $\mb{P}^1$. Since $\langle X,Y\rangle$ has zero points on $\mb{P}^1$, thus
  \begin{align}\label{complex.4}
    \langle X,Y\rangle\equiv 0,
  \end{align}
that is, the type-$(1,0)$ vector fields $X,Y$ are orthogonal to each other and has only finite number of common zero points on $M$. Now by the assumption of positive orthogonal bisectional curvature, we get
  \begin{align}
    \frac{\p^2 E(f)}{\p t\p\b{t}}=-2\int_{\mb{P}^1}|v|^{-2} R(X,\b{X},Y,\b{Y})\sqrt{-1}dw\wedge d\b{w}<0.
  \end{align}
On the other hand, since $f$ is energy minimal, we get $\frac{\p^2 E(f)}{\p t\p\b{t}}\geq 0$, which gives a contradiction.
Thus $f$ must be a holomorphic or conjugate holomorphic map.
\end{proof}

\section{The Proof of Theorem 0.1}

In this section, we prove Theorem 0.1 by using Siu-Yau's method. To do this, we first give following two propositions.

The first proposition has been proved in \cite{Gu}. However, for reader's convenience, we present a direct proof of it, in which we will use the Bishop-Goldberg's original statements (cf. \cite{Bishop} or \cite{Siu}, Theorem 3) to deal with the part of parallel real harmonic forms.
\begin{prop}\label{Pro 1}
  Let $(M^n,h)$, $n\geq 2$, be a compact K\"{a}hler manifold with positive orthogonal bisectional curvature. Then the second betti number
  $$b_2(M)=1$$
  and the first Chern class $c_1(M)$ is positive.
\end{prop}
\begin{proof}
  We first prove that $b_{1,1}(M):=\dim H^{1,1}(M,\mb{R})=1$. To do this it suffices to show that every real harmonic $(1,1)$-form $\alpha$ is a real multiple of the K\"{a}hler form $\omega_M$ of $(M,h)$.

  Denote $E:=T^*M$,  the K\"{a}hler metric $h$ induces a Hermitian metric $h^E$ on the holomorphic vector bundle $E$. Let $\n=\n^{1,0}+\b{\p}$ be the Chern connection of the Hermitian holomorphic vector bundle $(E,h^E)$. Let $(\n^{1,0})^*$ (resp. $\b{\p}^*$) be the adjoint operator of $\n^{1,0}$ (resp. $\b{\p}$)  with respect to the K\"{a}hler metric $h$ and the Hermitian metric $h^E$.  Set
  $$\t'=\n^{1,0}(\n^{1,0})^*+(\n^{1,0})^*\n^{1,0},\quad \t''=\b{\p}^*\b{\p}+\b{\p}\b{\p}^*.$$
  Suppose that
  $$\alpha=\sum_{i,j}\alpha_{i\b{j}}d\b{z}^j\wedge dz^i\in A^{0,1}(M,E).$$
  Since $\alpha$ is harmonic, by using Bochner-Kodaira-Nakano identity (cf. \cite{Dem}, Chapter VII, \S 1), we have
  \begin{align}\label{0}
  \begin{split}
    0=\int_{M}\langle \t''\alpha,\alpha\rangle dV_{\omega_M}&=\int_{M}\langle \t'\alpha,\alpha\rangle dV_{\omega_M}+\int_{M}\langle [\sqrt{-1} R^E,\Lambda]\alpha,\alpha\rangle dV_{\omega_M}\\
    &=\int_M |\n^{1,0}\alpha|^2 dV_{\omega_M}+\int_{M}\langle [\sqrt{-1} R^E,\Lambda]\alpha,\alpha\rangle dV_{\omega_M}.
    \end{split}
  \end{align}
   where $dV_{\omega_M}=\frac{\omega_M^n}{n!}$ and $\Lambda$ is the dual operator of the wedge product by $\omega_M$.

   From $R^E=[\n^{1,0},\b{\p}^E]$, we have locally
   \begin{align}
     \langle [\sqrt{-1} R^E,\Lambda]\alpha,\alpha\rangle
     =h^{\b{l}s}h^{\b{j}k}h^{\b{m}i}\alpha_{k\b{l}}\alpha_{i\b{j}}R_{s\b{m}}-h^{\b{l}s}h^{\b{j}k}h^{\b{n}m}h^{\b{t}i}\alpha_{k\b{l}}
     \alpha_{i\b{n}}R_{s\b{t}m\b{j}}.
   \end{align}
   Fix any point $p\in M$, we can choose a local coordinate system around $p$ such that $h_{i\b{j}}=\delta_{ij}$ and $\alpha_{i\b{j}}=a_i\delta_{ij}$ at $p$. Then the positivity of orthogonal bisectional curvature implies that
   \begin{align}\label{1.1}
   \langle [\sqrt{-1} R^E,\Lambda]\alpha,\alpha\rangle(p)=\sum_{i<j}R_{i\b{i}j\b{j}}(a_i-a_j)^2\geq 0.
   \end{align}
   Now by (\ref{0}), we get
   $$\n^{1,0}\alpha=0\quad\text{and}\quad  \langle [\sqrt{-1} R^E,\Lambda]\alpha,\alpha\rangle=0,$$
   and so
   $$\alpha=a\omega_M,$$
   for some real smooth function $a$. Therefore, we obtain $\p a=0$ from $\n^{1,0}\alpha=0$ and $\n^{1,0}\omega=0$, and hence
   $da=0$ by $a$ being real, that means that $a$ is a constant. Finally we get
   \begin{align}\label{1.7}
   b_{1,1}(M)=1.
   \end{align}

    Now we prove that $c_1(M)>0$. Suppose $\{e_{\alpha}\}$ is an orthonormal basis of $T^{1,0}M$ at one point $p\in M$, then for any $\alpha\neq \beta$, we have
  \begin{align}\label{complex.1}
    R(e_{\alpha}-e_{\beta},\o{e_{\alpha}}-\o{e_{\beta}},e_{\alpha}+e_{\beta},\o{e_{\alpha}}+\o{e_{\beta}})=R_{\alpha\b{\alpha}\alpha\b{\alpha}}+R_{\beta\b{\beta}\beta\b{\beta}}-
    R_{\alpha\b{\beta}\alpha\b{\beta}}-R_{\beta\b{\alpha}\beta\b{\alpha}}>0.
  \end{align}
  Similarly, changing $e_{\beta}$ by $\sqrt{-1}e_{\beta}$, we have
  \begin{align}\label{complex.2}
    R_{\alpha\b{\alpha}\alpha\b{\alpha}}+R_{\beta\b{\beta}\beta\b{\beta}}+
    R_{\alpha\b{\beta}\alpha\b{\beta}}+R_{\beta\b{\alpha}\beta\b{\alpha}}>0.
  \end{align}
 Combining (\ref{complex.1}) and (\ref{complex.2}), we obtain that
  \begin{align}\label{1.5}
  R_{\alpha\b{\alpha}\alpha\b{\alpha}}+R_{\beta\b{\beta}\beta\b{\beta}}>0.
  \end{align}
  By (\ref{1.5}), we have the following computations on the scalar curvature $R$ of $(M,h)$,
  \begin{align}\label{1.6}
  \begin{split}
  R&=\sum_{\alpha,\beta}R_{\alpha\b{\alpha}\beta\b{\beta}}
  =\sum_{\alpha}\sum_{\beta\neq\alpha}R_{\alpha\b{\alpha}\beta\b{\beta}}+\sum_{\alpha}R_{\alpha\b{\alpha}\alpha\b{\alpha}}\\
  &=\sum_{\alpha}\sum_{\beta\neq\alpha}R_{\alpha\b{\alpha}\beta\b{\beta}}+\frac{1}{2}\left(\sum_{\alpha=1}^nR_{\alpha\b{\alpha}\alpha\b{\alpha}}+\sum_{\beta=1}^nR_{\beta\b{\beta}\beta\b{\beta}}\right)\\
  &=\sum_{\alpha}\sum_{\beta\neq\alpha}R_{\alpha\b{\alpha}\beta\b{\beta}}+\frac{1}{2}\sum_{\gamma=1}^{n-1}(R_{\gamma\b{\gamma}\gamma\b{\gamma}}+R_{\gamma+1\o{\gamma+1}\gamma+1\o{\gamma+1}})+\frac{1}{2}(R_{n\b{n}n\b{n}}+R_{1\b{1}1\b{1}})\\
  &> 0.
  \end{split}
  \end{align}
By (\ref{1.7}) and $c_1(M)\in H^{1,1}(M,\mb{R})$, we can assume that for some constant $k$,
$$c_{1}(M)=k\omega_M.$$
So we have
$$k\int_M\omega_M^n=\int_{M}c_1(M)\wedge\omega_M^{n-1}=\frac{1}{n}\int_M R\omega_M^n>0,$$
and so $k>0$ and $c_1(M)>0.$
Now by Kodaira vanishing theorem, we get $H^2(M,\mc{O}_M)=0$, i.e., $b_{0,2}(M)=0$. Therefore, $b_2(M)=b_{1,1}(M)+2b_{0,2}(M)=1$.
\end{proof}

\begin{lemma}\label{prop.3}
   Let $M$ be a compact K\"{a}hler manifold with positive orthogonal bisectional curvature. Let $f:\mb{P}^1\to M$ be a nontrivial holomorphic map, and $L$ be a holomorphic quotient line bundle of $f^*TM$.
   Then $c_1(L)>0$.
\end{lemma}
\begin{proof}
  Since holomorphic line bundles over $\mb{P}^1$ are all of the form $\mc{O}(k)$, $k\in \mb{Z}$, we assume by contradiction that $L=\mc{O}(k)$, $k\leq 0$. Then $L^*=\mc{O}(-k)\geq 0$.  It follows that there exists a nonzero  holomorphic section  $s=s_{\alpha}dz^{\alpha}$ of $L^*$.
  Since $L$ is a holomorphic quotient line bundle of $f^*TM$, it follows that  $L^*$ is a holomorphic subbundle of $f^*T^*M$.
Then
\begin{align}
\begin{split}
\langle\sqrt{-1}f^*R^*s,s\rangle&= \sqrt{-1}(R^*)^{\alpha\b{\beta}}_{~~\gamma\b{\delta}}s_{\alpha}\o{s_{\beta}}f^{\gamma}_w\o{f^{\delta}_w}dw\wedge d\b{w}\\
&=\sqrt{-1}(-h^{\b{\tau}\alpha}h^{\b{\beta}\sigma}R_{\sigma\b{\tau}\gamma\b{\delta}})s_{\alpha}\o{s_{\beta}}f^{\gamma}_w\o{f^{\delta}_w}dw\wedge d\b{w}\\
&=-\sqrt{-1}R_{\sigma\b{\tau}\gamma\b{\delta}}(h^{\b{\beta}\sigma}\o{s_{\beta}})(\o{h^{\b{\alpha}\tau}\o{s_{\alpha}}})f^{\gamma}_w\o{f^{\delta}_w} dw\wedge d\b{w}\\
&=-\sqrt{-1}R_{\sigma\b{\tau}\gamma\b{\delta}}(h^{\b{\beta}\sigma}\o{s_{\beta}})(\o{h^{\b{\alpha}\tau}\o{s_{\alpha}}})vf^{\gamma}_w\o{vf^{\delta}_w} \frac{1}{|v|^2}dw\wedge d\b{w}\\
&=-\sqrt{-1}R(s^*,\overline{s^*},X,\overline{X})\frac{1}{|v|^2}dw\wedge d\b{w},
\end{split}
\end{align}
where  $X:=v\frac{\p f^{\alpha}}{\p w}\frac{\p}{\p z^{\alpha}}$ which is defined by (\ref{XY}), $s^*:=h^{\b{\alpha}\tau}\o{s_{\alpha}}\frac{\p}{\p z^{\tau}}$.

Note that
\begin{align}\label{complex.5}
\left\langle s^*,X\right\rangle\equiv 0,
\end{align}
where the proof of (\ref{complex.5}) is the same as (\ref{complex.4}). By  the positive orthogonal bisectional curvature condition, one has
\begin{align}
  \langle\sqrt{-1}f^*R^*s,s\rangle=-\sqrt{-1}R(s^*,\overline{s^*},X,\overline{X})\frac{1}{|v|^2}dw\wedge d\b{w}<0
\end{align}
 almost everywhere on $\mb{P}^1$.

 Since the curvature is decreasing on  holomorphic subbundle, thus
  \begin{align}
   -k=\int_{\mb{P}^1}c_1(L^*)\leq \int_{\mb{P}^1}\langle\sqrt{-1}f^*R^*s,s\rangle<0,
 \end{align}
  which  contradicts to the assumption $k\leq 0$. Thus $c_1(L)>0$.
\end{proof}

The following proposition is an analogue of Siu-Yau's Proposition 3 in \cite{Siu}.
\begin{prop}\label{prop.4}
  Let $M$ be a compact K\"{a}hler manifold with positive orthogonal bisectional curvature. Let $C_0$ be a rational curve in $M$ and $f_0:\mb{P}^1\to C_0$ be its normalization.
  Then there exists a proper subvariety $Z$ of $P(TM)$ with the following property. If $y\in M$ and $\xi\in (TM)_y-0$ define an element of $P(TM)-Z$, then there exists a holomorphic
  map $f:\mb{P}^1\to M$ homotopic to $f_0$ ( when $f_0$ is regarded as a map from $\mb{P}^1$ to $M$ ) such that $y$ is a regular point of $f(\mb{P}^1)$ and the tangent vector of $f(\mb{P}^1)$ at
  $y$ is  a  nonzero multiple of $\xi$.
\end{prop}
\begin{proof}
  The proof is basically the same as Proposition 3 in \cite{Siu} except at two places where they required $TM$ to be positive in the sense of Griffiths, both of which follow from Lemma \ref{prop.3}.

  The first one is the proof of $H^1(V, N_V)=0$:

 Let $V\subset \mb{P}^1\times M$ denote the graph of $f_0$, where $f_0$ is holomorphic map. Let $\pi:\mb{P}^1\times M\to\mb{P}^1$, $\sigma:\mb{P}^1\times M\to M$ be the natural projections. Since $T({\mb{P}^1}\times M)$ is isomorphic to
  $\pi^*T\mb{P}^1\oplus\sigma^* TM$ and $TV$ is isomorphic to $\pi^*T\mb{P}^1$, it follows that $\sigma^*TM|V$ is isomorphic to the normal bundle $N_V$ of $V$ in $\mb{P}^1\times M$, that is,
  $f_0^* TM$ is isomorphic to $N_V$. By a theorem of Grothendieck \cite{Gro}, $f_0^*TM$  splits into a sum of holomorphic line bundles $L_1,\ldots, L_n$. By Lemma \ref{prop.3}, each
  $L_i$ is a positive holomorphic line bundle over $V$. Then by the theorem of Riemann-Roch, $H^1(V,L_i)=0$ for each $i$, then $H^1(V,f^*_0 TM)=0$. Since $f^*_0TM$ is isomorphic to $N_V$, it follows that
  $H^1(V,N_V)=0$.

  The second one is that when $f^*_0TM/(T\mb{P}^1\otimes [E])$ splits into a direct sum of holomorphic line bundles $Q_2,\ldots,Q_m$ over $\mb{P}^1$, then each $Q_i$ is positive line bundle, and this fact also follows from Lemma \ref{prop.3} directly.
\end{proof}
\bigskip
We are now ready to prove Theorem 0.1.\\
\vskip 3mm
\emph{ Proof:}
By Proposition \ref{Pro 1},  $M$ is a compact K\"{a}hler manifold with $c_1(M)>0$, by a theorem of Yau \cite{Yau}, there exists a K\"{a}hler metric with positive Ricci form and so $M$ is simply connected by Theorem A of Kobayashi in \cite{Kob}.
Thus $\pi_2(M)$ and $H_2(M,\mb{Z})$ are isomorphic. On the other hand,  $b_2(M)=1$ by Proposition \ref{Pro 1}. Now by the universal coefficient theorem, we have $H^2(M,\mb{Z})\cong \mb{Z}$. Hence, there exists a positive holomorphic line bundle $F$ over $M$ whose first Chern class $c_1(F)$ is a generator of $H^2(M,\mb{Z})$. Let $g$ be the generator of the free part of $H_2(M,\mb{Z})$ such that
$c_1(F)[g]=1$. Since $H_2(M)$ is isomorphic to $\pi_2(M)$, we regard $g$ also as an element of $\pi_2(M)$. By the famous theorem of Sacks-Uhlenbeck (cf. \cite{Sacks}, \cite{Siu}), there exist minimal harmonic maps
$f_i:\mb{P}^1\to M$, $0\leq i\leq k$, such that
\begin{enumerate}
  \item the element of $\pi_2(M)$ is defined by $\sum_{i=0}^k f_i$ is $g$ and
  \item $\sum_{i=0}^k E(f_i)$ equals the infimum of $E(h)$ for all maps $h:\,\mb{P}^1\to M$ which define the element $g$ in $\pi_2(M)$.
\end{enumerate}

By Theorem \ref{thm.1}, each $f_i\, (0\leq i\leq k)$ is either holomorphic or conjugate holomorphic. So each $f_i(\mb{P}^1)(0\leq i\leq k)$ is a rational curve. Since $c_1(TM)$ is a positive integral multiple of
$c_1(F)$ and the value of $c_1(F)$ at $g$ is 1, it follows that at least one $f_i$ is holomorphic. If $k>0$, then at least one $f_j$ is conjugate holomorphic. We distinguish now between two cases.

{\bf Case 1.} $k=0$.

The line bundle $T\mb{P}^1\otimes [E]$ is a subbundle of $f^*_0 TM$ and the quotient bundle $(f^*_0TM)/(T\mb{P}^1\otimes[E])$ splits into a direct sum of line bundles $Q_2,\ldots,Q_n$, where $E$ is the divisor of the differential $df_0$ of $f_0$ and $[E]$
is the corresponding line bundle over $\mb{P}^1$ associated to the divisor $E$. By Lemma \ref{prop.3}, each $Q_i\, (2\leq i\leq n)$ is a positive line bundle. It follows that
$$c_1(f^*_0TM)=c_1(\mb{P}^1)+c_1([E])+\sum_{i=2}^nc_1(Q_i).$$
Hence $c_1(TM)$ evaluated at $g$ is $\geq n+1$. That is, $c_1(TM)=\lambda c_1(F)$ for some integer $\lambda\geq n+1$. By the result of Kobayashi-Ochiai \cite{Kob1},
$M$ is biholomorphic to $\mb{P}^n$.

{\bf Case 2.} $k>0$.
Without loss of generality we can assume that $f_0$ is holomorphic and $f_1$ is conjugate holomorphic.
By Proposition \ref{prop.4} and the same proof as Siu-Yau \cite{Siu}, \S 6, we get a contradiction. Hence $k=0$ and $M$ is biholomorphic to $\mb{P}^n$.

\end{document}